\documentclass[english,12pt]{amsart}

\usepackage{graphicx,epsfig,psfrag}
\usepackage{amssymb,amsfonts,amsmath,amsthm,dsfont,wasysym,pifont,stmaryrd}
\usepackage{epstopdf,yfonts}
\usepackage{hyperref}
\usepackage{psfrag}
\input pstricks
\input pst-node

\usepackage{pstricks, pst-node}
\input xy
\usepackage[all]{xy}

\newcommand{\thismonth}{\ifcase\month 
  \or January\or February\or March\or April\or May\or June%
  \or July\or August\or September\or October\or November%
  \or December\fi}

\newcommand{\mcA}{\mathcal{A}}

\newcommand{\mcT}{\mathcal{T}}
\newcommand{\mcC}{\mathcal{C}}

\newcommand{\mcH}{\mathcal{H}}

\newcommand{\mcL}{\mathcal{L}}

\newcommand{\mcU}{\mathcal{U}}

\newcommand{\one}{\mathbf{1}}

\newcommand{\Ff}{\mathbb{F}}

\newcommand{\om}{\omega}

\newcommand{\Hind}{\mcH_{\textrm{ind}}}
\newcommand{\piind}{\pi_{\textrm{ind}}}

\DeclareMathOperator{\Aut}{Aut}

\newcommand{\inv}[1]{{#1}^{-1}}

\newcommand{\norm}[1]{\|{#1}\|}

\newcommand{\defn}[1]{\emph{#1}}
\newcommand{\ind}[1]{\operatorname{Ind}_{\Gamma_0}^\Gamma {(#1)}}
\newcommand{\st}{\,:\;} 
\newcommand{\period}{\; .}    

\title[Representations of Virtually Free Groups]{A New Family of
Representations \\ of Virtually Free Groups}

\author{Alessandra Iozzi}
\address{Departement Mathematik\\
ETH Z\"urich\\
8092 Z\"urich\\
SWITZERLAND }
\email{iozzi@math.ethz.ch}

\author{M. Gabriella Kuhn}
\address{Dipartimento di Matematica \
Universit\`a di Milano ``Bicocca''\\
Via Cozzi 53\\ 
20125 Milano, ITALIA}
\email{mariagabriella.kuhn@unimib.it}

\author{Tim Steger}
\address{Facolt\`a di Scienze Matematiche Fisiche e Naturali\\
Universit\`a degli Studi di Sassari\\
Via Piandanna 4\\
07100 Sassari, ITALIA}
\email{steger@uniss.it}


\date{\today}

\newtheorem{theorem}{Theorem}[section]
\newtheorem{theorem_intro}{Theorem}
\newtheorem{corollary}[theorem]{Corollary}

\newtheorem{proposition}[theorem]{Proposition}

\theoremstyle{definition}
\newtheorem{definition}[theorem]{Definition}
\newtheorem{remark}[theorem]{Remark}

\newtheorem{example}[theorem]{Example}

\numberwithin{equation}{section}

\begin{document}

\parindent=1em\

\pagestyle{myheadings}

\begin{abstract} 
We construct a new family of irreducible unitary representations of a finitely generated
virtually free group $\Lambda$. We prove furthermore 
a general result concerning representations of Gromov hyperbolic groups 
that are weakly contained in the regular representation, thus implying 
that all the representations in this family can be realized on the boundary of $\Lambda$.
As a corollary, we obtain an analogue of Herz majorization principle.
\end{abstract}

\maketitle

\pagestyle{myheadings}

\numberwithin{equation}{section}

\section{Introduction}

Free groups are ubiquitous in mathematics and their representation
theory has been widely studied.  However, since a (finitely generated)
free group $\Gamma$ is not type I, the usual program of representation
theory in its na\"\i ve form, decomposing unitary representations into
irreducible ones, is almost meaningless.  In fact, to construct a
unitary representation of $\Gamma$ it is only necessary to fix a Hilbert
space $H$ and to choose a unitary operator for each
generator. 
A ``random'' choice will yield an irreducible representation.

If we restrict our
attention to those representations that are weakly contained in the regular
representation the situation drastically changes. For brevity we shall say 
that a representation is {\em tempered} if it is weakly contained in the regular representation. 
Using the fact that the reduced $C^\ast$ algebra of $\Gamma$ is simple (\cite{Pow}),
one can prove \cite{K-S2} that a tempered representation $(\pi,H)$ can always be realized
as a {\em boundary representation} (see \S\ref{sec:boundary-rep} for the definition).
This implies in particular that the Hilbert space $H$ can be chosen
to be  a direct integral of a measurable field of Hilbert spaces
$H=\int_{\partial\Gamma}^\oplus H_xd\mu(x)$ over the 
boundary $\partial\Gamma$ of $\Gamma$ for a suitable quasi-invariant measure $\mu$
which depends on the representation.

In 2004, 
a large family of irreducible unitary tempered representations of the free group,
the so-called {\it multiplicative} representations, was introduced \cite{K-S3}.
Although these representations have a very concrete and seemingly elementary definition,
this family is large enough to include all known specific irreducible
tempered representations constructed using  the action of $\Gamma$ on its Cayley graph. 

In \cite{Iozzi_Kuhn_Steger_stab} 
we extended the class in \cite{K-S3} to include also representations 
that are obtained with a similar procedure as in \cite{K-S3}
but are  only {\it finitely reducible}.
This has the advantage that this enlarged class of representations,
called the class $\mathbf{Mult}(\Gamma)$, is now stable under many natural operations,
such as the restriction to a subgroup and the induction to a free supergroup
(see \cite{Iozzi_Kuhn_Steger_stab}).  
Moreover, 
although the construction presented in \cite{K-S3}
seems to depend on the choice of generators, the class $\mathbf{Mult}(\Gamma)$  
is independent on that choice and in fact 
it is invariant under the action of $\Aut(\Gamma)$.
This fact is not true for 
example for the restriction to the free group in two generators
of the spherical series of the group of automorphisms of the homogeneous tree of valency four.
(See Remark~\ref{rem:facts}(2) for more 
on the irreducibility of these representations.)

\medskip
In this paper, in analogy with the case of the free group, we define
a new class of representations for virtually free groups.
These groups include for example  $\mathrm{PSL}(2,\mathbb Z)\cong\mathbb Z_2\ast \mathbb Z_3$,
whose commutator subgroup  is a torsion-free surface group and whose abelianization
$\mathrm{PSL}(2,\mathbb Z)_\mathrm{ab}\cong\mathbb Z_2\times\mathbb Z_3$ has order six.  
Furthermore, virtually free groups are Gromov hyperbolic and can be realized 
as fundamental groups of finite graph of finite groups, \cite{Karrass_Pietrowski_Solitar}.

We define a class $\mathbf{Mult}(\Lambda)$ of unitary representations
of a finitely generated virtually free group $\Lambda$ by inducing a
representation of the class $\mathbf{Mult}(\Gamma)$ from a (in fact,
any) free subgroup $\Gamma$ of finite index (see \S~\ref{sec:lambda}) .   
For these classes of representations we prove the following

\begin{theorem_intro} Let $\Lambda$ be a virtually free group.  
\begin{enumerate}
\item The classes $\mathbf{Mult}(\Lambda)$ and $\mathbf{Mult}_{\mathrm irr}(\Lambda)$
are non-empty and $\Aut(\Lambda)$-invariant (Corollary~\ref{cor:invariance}).
\item The representations in the class $\mathbf{Mult}(\Lambda)$ are weakly contained in the 
regular representation (Corollary~\ref{cor:bdry-rep-Gamma}).
\item The class $\mathbf{Mult}(\Lambda)$ are cocycle representations of $\Lambda$,
that is representations of the cross product
$\Lambda\ltimes\mcC(\partial\Lambda)$ (Corollary~\ref{cor:bdry}). 
\end{enumerate}
\end{theorem_intro}


As we mentioned earlier, the representations of the class
$\mathbf{Mult}(\Gamma)$ encompass all tempered representations of the
free group $\Gamma$ that arise from the embedding of $\Gamma$ into
the group of automorphisms of its Cayley graph (with respect to some
set of generators). On the other hand, to the authors' knowledge, we
are not aware of other realizations of any of the representations in
the class $\mathbf{Mult}_\mathrm{irr}(\Lambda)$ of a virtually free
group $\Lambda$.  Constructions similar to ours (in the cocompact case)
can be found for example in \cite{Bader_Muchnik}, where the authors
show the irreducibility of the quasi-regular representation of a
compact surface group $\pi_1(\Sigma)$ on the geodesic
boundary\footnote{One word of warning for the reader: what the authors
  in \cite{Bader_Muchnik} call "boundary representation" is not what
  is referred to with the same terminology in this paper, but what we call 
``quasi-regular'' representation in Theorem~\ref{prop:herz}.}
$\partial\Sigma$ with respect to the Patterson--Sullivan measure.
Likewise, in \cite{Burger_delaHarpe}, the authors show that if
$H<L$ are discrete groups such that
$H=\mathrm{Comm}_L(H)$, then the induction to $L$ of
any finite dimensional irreducible representation of $H$ remains
irreducible.  None of these result seem to have a nonempty
intersection with our construction.

\medskip
The last item in the above theorem follows from a result that was already known for free groups and
we record here for a general Gromov hyperbolic group (see Theorem~\ref{general}), namely:

\begin{theorem_intro}
Let $ G $ be a torsion free Gromov hyperbolic group which is not almost cyclic. 
Then every tempered representation of $ G $ is a cocycle representation with respect to 
some quasi-invariant measure.  

If the representation is irreducible, the measure 
can be taken to be ergodic.
\end{theorem_intro}

As a consequence of this result we prove the following analogue of Herz
majorization principle:

\begin{theorem_intro}
Let  $(\pi,H)$ be a tempered 
representation of  a torsion free Gromov hyperbolic group $ G $  
which is not almost cyclic  
and let $v$ be any vector in $H$. 
Then there exists a positive measure $\mu$ on $\partial G $ such that
\begin{equation*}
|\langle\pi(x)v,v\rangle|\leq \norm{v}^2|\langle\rho(x)\one_{\partial G },\one_{\partial G }\rangle|
\end{equation*}
where $\rho$ is the quasi-regular representation on $L^2(\partial G ,d\mu)$ and 
$\one_{\partial G }$ is the constant function on $\partial G $.
\end{theorem_intro}

The measure on $\partial G$ must however depend on the tempered
representation, thus implying that a Harish-Chandra function cannot
exist (see Remark~\ref{rem:4.8}) and exhibiting one more instance of
the fact that Gromov hyperbolic group behave morally as rank one
groups.

\medskip
We remark that the above construction relies not only upon the
stability properties of the class $\mathbf{Mult}(\Gamma)$ of a free
group $\Gamma$ (which were proven in \cite{Iozzi_Kuhn_Steger_stab}),
but also of the non-obvious corresponding properties of the extension
of multiplicative representations to boundary representations (see for
example Theorem~\ref{thm:ind}).



\bigskip The structure of the paper is as follows.  In
\S~\ref{sec:prelim} we recall the definition of boundary
representation of a free group -- and, more generally, of a Gromov
hyperbolic group -- and the construction of the boundary
multiplicative representations of the free group; we recall moreover
from \cite{Iozzi_Kuhn_Steger_stab} the stability properties of the
class of representations of the free group obtained from matrix
systems with an inner product.  In \S~\ref{sec:lambda} we define the
classes $\mathbf{Mult}(\Lambda)$ and $\mathbf{Mult}_\mathrm{irr}(\Lambda)$ of 
representations of a finitely generated virtually free group $\Lambda$
obtained by induction from any free subgroup of finite index and we show
both that $\mathbf{Mult}(\Lambda)$ and
$\mathbf{Mult}_\mathrm{irr}(\Lambda)$ are $\Aut(\Lambda)$-invariant and
that these representations are tempered. In
\S~\ref{sec:result-hyperbolic} we prove that (irreducible) tempered
representations of a Gromov hyperbolic group $G$ are cocycle
representations with respect to an (ergodic) measure and we deduce the
analogue of Herz majorization
principle (Theorem~\ref{prop:herz}).  In the Appendix~\ref{app:2}
we prove the essential stability results for multiplicative boundary
representations that are not proven in \cite{Iozzi_Kuhn_Steger_stab}.

\section{Preliminaries}\label{sec:prelim}
\subsection{Boundary Representations}\label{sec:boundary-rep}

\begin{definition}
Let $G$ be any discrete group, $\mcA$ be a commutative $C^\ast$-algebra and
$\lambda:G\to \Aut(\mcA)$ a homomorphism of $G$ 
into the group of isometric automorphisms of $\mcA$. 
A \defn{covariant representation} of $(G,\mcA)$ on a Hilbert space $H$  is a triple
$(\pi,\alpha,\mcH)$ where
\begin{itemize}
\item
 $\pi:G\to\mcU(H)$ is a unitary representation of $G$;
\item $\alpha:\mcA\to\mcL(H)$ is a $C^\ast$-representation of $\mcA$
in the space of bounded linear operators; 
\item for all $\gamma\in G$ and $A\in\mcA$
\begin{equation*}
\pi(\gamma)\alpha(A)\pi(\inv\gamma)=\alpha\big(\lambda(\gamma)A\big)\,.
\end{equation*}
\end{itemize}
Two covariant representations $(\pi,\alpha,\mcH)$ and $(\rho,\beta,\mcL)$ of $G$ and $\mcA$
are {\em equivalent} if there exists a unitary operator $J:\mcH\to\mcL$,
such that for all $\gamma\in G$ and all $A\in\mcA$,
\begin{equation*}
\rho(\gamma)\,J=J\,\pi(\gamma)\qquad\text{ and }\qquad\beta(A)\,J=J\,\alpha(A)\,.
\end{equation*}
\end{definition}

\medskip
If $K$ is any compact metrizable space on which $G$ acts continuously and
by isometries, the space of complex valued functions $\mcC(K)$ is a $C^\ast$-algebra
naturally endowed with a continuous isometric action of $G$, 
$\lambda:G\to\Aut\big(\mcC(K)\big)$,
defined by 
\begin{equation*}
\lambda(\gamma)F(k):=F(\inv\gamma k)\,,
\end{equation*} 
for all $F\in \mcC(K)$, $\gamma\in G$ and $k\in K$.

In the case in which $G$ is a Gromov
hyperbolic group, the space $K$ can be taken to be the boundary 
of the Cayley graph associated to a fixed generating system, which we
denote by $\partial G$.  For the sake of the reader, we recall the definition of 
$\partial G$ in the appendix.  We
mention here only that $\partial G$ is a compact metrizable space with the
$G$-action defined by $(\gamma,\omega)\mapsto\inv\gamma\omega$
and that different generating sets correspond to homeomorphic
boundaries. 

\begin{definition} A {\em boundary representation} of a hyperbolic
  group $G$ on $H$ is a covariant representation $(\pi,\alpha,H)$
  of $\big(G,\mcC(\partial G)\big)$.
\end{definition}

The reader who is familiar with crossed-product $C^\ast$-algebras will
recognize that a boundary representation is nothing but a
representation of the crossed product $C^\ast$-algebra $G\ltimes\mcC(\partial G)$ 
(see \S~\ref{sec:result-hyperbolic}).

General Gromov hyperbolic groups will be considered again in their
full generality in \S~\ref{sec:result-hyperbolic}, while in the rest
of this section we will consider only free groups.

\subsection{Boundary Multiplicative Representations of the Free Group}
We begin with the definition of {\it multiplicative representation} 
in the context of finitely generated free groups, referring  to
\cite{K-S3}  for details and proofs.

Let $\Ff_A$ be a free group with a finite symmetric set of free
generators $A$.  A {\em matrix system} $(V_a,H_{ba})$, is an
assignment of a vector space $a~\mapsto~V_a$ for every generator $a\in
A$ and a linear map $H_{ba}:V_a\to V_b$, for every $a,b~\in~A$, such
that $H_{ba}=0$ whenever $ba=e$.  An {\em invariant subsystem}
$(W_a,H_{ba})$ of the matrix system $(V_a,H_{ba})$ is an assignment of
vector subspaces $a\mapsto W_a\subseteq V_a$ such that
$H_{ba}W_a\subset W_b$ for all $a,b\in A$.  If $(W_a,H_{ba})$ is an
invariant subsystem of $(V_a,H_{ba})$, the {\em quotient system}
$(\widetilde{V}_a,\widetilde{H}_{ba})$ is the assignment $a\mapsto
\widetilde{V}_a:=V_a/W_a$ such that
$\widetilde{H}_{ba}\widetilde{v}_{a}=H_{ba}v_a$, for any
representative $v_a$ of $\widetilde{v}_a\in\widetilde{V}_a$.  A system
$(V_a,H_{ba})$ is called {\em irreducible} if it is nonzero and admits
no nontrivial invariant subsystems.

We endow $\Ff_A$ with the word metric $d(x,e):=|x|$
with respect to the generating set $A$.
We say that a function 
\begin{equation*}
f:\Ff_A\to\bigsqcup_{a\in A} V_a
\end{equation*}
is {\em multiplicative} if there exists $N\geq0$, depending only on $f$,
such that for all $x$ with $|x|\geq N$
\begin{equation}\label{1.1}
\begin{alignedat}{3}
&f(x)\in V_a          &\quad&\text{ if }x=x'a\text{ is reduced}\\
&f(xb)=H_{ba}f(x)&\quad&\text{ if }x=x'a\text{ is reduced and }ba\neq e\,.
\end{alignedat}
\end{equation}
We denote by $\mcH^\infty(V_a,H_{ba})$ (or by $\mcH^\infty$ if no
confusion arises) the quotient of the space of multiplicative
functions with respect to the equivalence relation according to which
two multiplicative functions are equivalent if they differ only on
finitely many words.

If for every $a\in A$ the $V_a$'s are equipped with a positive
definite sesquilinear form $B_a$ and if these forms satisfy the {\em
  compatibility condition}
\begin{equation}\label{E-cond-B}
B_a(v_a,v_a)=\sum_{b\in A}B_b(H_{ba}v_a,H_{ba}v_a)
\end{equation}
for all $v_a\in V_a$,
then
\begin{equation}\label{1.2}
\langle f_1,f_2\rangle
:=\sum_{|x|=N}\;\;\sum_{ \substack{ \;a\\ |xa|=|x|+1}}
B_a\big(f_1(xa),f_2(xa)\big)
\end{equation}
defines an inner product on $\mcH^\infty$,
where $N$ should be taken to be large enough that both $f_1$ and $f_2$ satisfy \eqref{1.1}
outside the ball of radius $N$.  We remark that, up to a normalization, 
every matrix system $(V_a,H_{ba})$ admits a compatible tuple $(B_a)_{a\in A}$
of positive semidefinite forms. 
When the matrix system is irreducible,
then each $B_a$ is strictly definite positive and, up to multiple scalars, 
it is also unique. Whether the system is irreducible or not,
the triple $(V_a,H_{ba},B_a)$ will be called
a {\em matrix system with inner product}.
We can hence define a representation of $\Ff_A$ on $\mcH^\infty(V_a,H_{ba})$ by
\begin{equation*}\label{repr}
\big(\pi(x)f\big)(y):=f(\inv x y)\,,
\end{equation*}
which can be proved to be unitary. 
If $\mcH(V_a,H_{ba},B_a)$ is the completion of
$\mcH^\infty(V_a,H_{ba})$ with respect to the inner product in
\eqref{1.2}, then $\pi$ extends to a unitary representation on
$\mcH(V_a,H_{ba},B_a)$, which we called {\em multiplicative}.

\medskip

The next step is to show that multiplicative representations are
in fact boundary representations of the free group.  

The boundary $\partial\Ff_A$ of a free group $\Ff_A$ consists of the set of infinite
reduced words, with the topology defined by  the basis
\begin{equation*}
\partial\Ff_A(x):=\{\omega\in\partial\Ff_A:\,\text{the reduced word for }\omega\text{ starts with }x\}\,,
\end{equation*}
for all $x\in\Ff_A$, $x\neq e$.  The sets $\partial\Ff_A(x)$ are both open and
closed in $\partial\Ff_A$ and $\partial\Ff_A$ is a compact (as well as Hausdorff, perfect,
separable, and totally disconnected) space.  For every $x\in\Ff_A$,
$x\neq e$, let $\one_{\partial\Ff_A(x)}$ denote the characteristic function of
$\partial\Ff_A(x)$.  In order to show that a given unitary representation
$(\pi,\mcH)$ of $\Ff_A$ is a boundary representation we need to define
an algebra $C^\ast$-homomorphism $\alpha:\mcC(\partial\Ff_A)\to\mcL(\mcH)$
satisfying
\begin{equation}\label{prodinc}
\pi(x)\alpha(F)\pi(x^{-1})=\alpha\big(\lambda(x)F\big)\,,
\end{equation}
for any $x\in\Ff_A$ and~$F\in \mcC(\partial\Ff_A)$.

Since  the subalgebra spanned by the
functions $\{\one_{\partial\Ff_A(x)}\}_{x\in\Ff_A}$  is a dense $C^\ast$-subalgebra of $\mcC(\partial\Ff_A)$,
it is sufficient to define 
$\alpha_\pi(\one_{\partial\Ff_A(x)})$ for every $x$, and in fact  on the dense subspace $\mcH^\infty\subset\mcH$.
Denote by $\one_{\Ff_A(x)}$ the characteristic function
of the cone
\begin{equation}\label{eq:fx}
\Ff_A(x):=\{y\in\Ff_A:\,\text{the reduced word for }y\text{ starts with }x\}
\end{equation}
and define $\alpha_\pi(\one_{\partial\Ff_A(x)}):\mcH^\infty\to
\mcH^\infty$ 
by setting
\begin{equation}\label{eq:bdryrep}
\big(\alpha_{\pi}(\one_{\partial\Ff_A(x)})f\big)(y):=\one_{\Ff_A(x)}(y)f(y)=
\begin{cases}
f(y) &\text{if $y\in \Ff_A(x)$}\\
0    &\text{otherwise}\,.
\end{cases}
\end{equation}
A routine calculation shows that \eqref{prodinc} is verified and hence every
multiplicative representation $(\pi,\mcH)$ 
is a boundary representation
$(\pi,\alpha_\pi,\mcH)$ of $\Ff_A$.
\begin{remark}\label{rem:facts}
\begin{enumerate}
\item When a boundary representation  is considered as a representation
  of $\Ff_A$ it is always weakly contained in the regular
  representation. This follows from general considerations 
  since $\Ff_A$ acts amenably on $\partial\Ff_A$; a two pages proof specifically  for the
  case of the free group can be found in \cite[\S~2]{K-S1}.
\item In \cite{K-S3} it is shown that multiplicative representations built 
from an irreducible system are irreducible as boundary representations, (that is
as representations of  the cross-product $\Ff_A\ltimes \mcC(\partial\Ff_A)$, see \S~\ref{sec:result-hyperbolic})
while, as representations of $\Ff_A$, they are either irreducible or, in some
special cases, are sum of two irreducible nonequivalent representations.
\end{enumerate}
\end{remark}

\subsection{Stability Properties of Boundary Multiplicative Representations}\label{sec:stability}
The definition of multiplicative representation
seems to depend on the generating set $A$ that we have fixed. We shall see that
the dependence  is only apparent, as soon as we allow general (not only irreducible)
matrix systems. The advantage of considering general matrix systems is that 
the new class of  representations so obtained is closed under change of generators, 
restriction and induction. The price to pay is not so high, as the following result shows:

\begin{theorem}[\cite{Iozzi_Kuhn_Steger_stab}]\label{thm:decomposition}
If $\pi$ is a representation constructed from a matrix system with inner product $(V_a,H_{ba},B_a)$,
then $\pi$ decomposes as an orthogonal direct sum with respect to $(B_a)_{a\in A}$
of a finite number of representations defined from irreducible matrix
systems and the same is true when $\pi$ is considered as a boundary representation.
\end{theorem}

We proceed now to infer further properties of multiplicative
representations of $\Ff_A$.

\begin{theorem}[\cite{Iozzi_Kuhn_Steger_stab}]
\label{thm:stab}
Let $\Ff_A$ be a group freely generated by the symmetric set $A$
and let $\big(\pi,\alpha_\pi,\mcH(V_a,H_{ba},B_a)\big)$
be a multiplicative boundary representation constructed from a matrix
system with inner products $(V_a,H_{ba},B_a)_{a\in A}$.
If $A'$ is another symmetric set of free generators 
such that $\Ff_A\cong\Ff_{A'}$, 
then there exists a multiplicative boundary representation
$\big(\pi',\alpha_{\pi'},\mcH(V_s,H_{ts},B_s)\big)$ constructed from
a matrix system with inner products $(V_s,H_{ts},B_s)_{s\in A'}$, such that
$\big(\pi,\alpha_\pi,\mcH(V_a,H_{ba},B_a)\big)$
appears either as a subrepresentation or as a quotient 
of $\big(\pi',\alpha_{\pi'},\mcH(V_s,H_{ts},B_s)\big)$.
\end{theorem}

We can therefore denote a free group by $\Gamma$ without any explicit dependence
on a free generating set.

We warn the reader that there is no guarantee that changing generators will
preserve the irreducibility of the system:
in \cite{Iozzi_Kuhn_Steger_stab} it is shown 
that a representation of the principal series for the free group can be realized
as a multiplicative representation from an irreducible matrix system, but, 
once the simplest nontrivial change of generator is performed, 
it arises from a quotient of a reducible matrix system.

\begin{theorem}[\cite{Iozzi_Kuhn_Steger_stab}]\label{thm:4.4}
Let  $\Gamma_0 \leq \Gamma$ be  a subgroup of finite index in the free group 
$\Gamma$.
Then:
\begin{enumerate}
\item the restriction to $\Gamma_0$ of a multiplicative boundary representation 
$(\pi,\alpha_\pi,\mcH)$ of $\Gamma$
 is a boundary multiplicative representation  of $\Gamma_0$;
\item if $(\pi',\alpha_{\pi'}',\mcH)$ is a boundary multiplicative representation of $\Gamma_0$, 
 then the induced representation  $\ind{\pi'}$ is a boundary multiplicative representation
 of $\Gamma$.
\end{enumerate}
\end{theorem}

Strictly speaking, the theorems stated in this section are proved in \cite{Iozzi_Kuhn_Steger_stab}   
when all the representations involved are considered only as representations of the free group
rather then boundary representations.
The extension of these results to the case of boundary representations is,
in most of the cases, a straightforward verification.  
The one that is a bit more involved is the proof of Theorem~\ref{thm:4.4}(2):
since it uses heavily the notations and the techniques of \cite{Iozzi_Kuhn_Steger_stab}, 
we defer it to the appendix of this paper.

The above theorems lead to the following:
\begin{definition} Let $\Gamma$ be a finitely generated  free group.
A representation $\rho:\Gamma\to\mcU(H)$ is in the class $\mathbf{Mult}(\Gamma)$
if there exist a symmetric  set $A$ of free generators,
a matrix system with inner product $(V_a,H_{ba},B_a)$, 
a dense subspace $M\subset H$
and a unitary operator $J:H\to\mcH(V_a,H_{ba},B_a)$ such that
\begin{enumerate}
\item $J$ is an isomorphism between $M$ and $\mcH^\infty(V_a,H_{ba}, B_a)$, and
\item for all $m\in M$ and $x\in\Gamma$, $J\big(\rho(x)m\big)=\pi(x)(Jm)$, 
where $\pi$ is the multiplicative representation constructed from
$(V_a,H_{ba},B_a)$.
\end{enumerate}
\end{definition}

\section{The Classes  $\mathbf{Mult}(\Lambda)$ and $\mathbf{Mult}_\mathrm{irr}(\Lambda)$}\label{sec:lambda}
Let $\Lambda$ be a finitely generated virtually free group.
\begin{definition}
We say that a representation $\pi$
of $\Lambda$ belongs to the class ${\mathbf{Mult}_0(\Lambda)}$
if it is contained in a representation obtained  inducing a
representation of the class $\mathbf{Mult}(\Gamma_0)$ where $\Gamma_0$ is
a free subgroup of finite index in  $\Lambda$. In other
words
\begin{equation*}
{\mathbf{Mult}_0(\Lambda)}:=\big\{\pi\in\Lambda\st \exists\,
\pi'\in\mathbf{Mult}(\Gamma_0)\st
\pi\leq \operatorname{Ind}_{\Gamma_0}^\Lambda(\pi')\big\}
\end{equation*}
\end{definition}

\begin{proposition}\label{prop:independence} The class ${\mathbf{Mult}_0(\Lambda)}$ does not depend on $\Gamma_0$.
\end{proposition}
\begin{proof}  Let  $\Gamma_1$ be another free group of finite index in $\Lambda$ and
${\mathbf{Mult}_1(\Lambda)}$ be the corresponding class of induced representations.
The stabilizer of the pair $\Gamma_0\times \Gamma_1\in \Lambda/{\Gamma_0}\times \Lambda/{\Gamma_1}$ 
for the diagonal action of $\Lambda$ is $\Gamma_0\cap \Gamma_1$. 
Hence $\Lambda/\Gamma_0\cap \Gamma_1$,  
as well as $\Gamma_0/\Gamma_0\cap \Gamma_1$ and $\Gamma_1/\Gamma_0\cap \Gamma_1$, are finite. 
Assume now that $\pi\in{\mathbf{Mult}_0(\Lambda)}$.  By definition there exists a representation 
$\pi_0$ of $\Gamma_0$ such that $\pi$ is a component of
$\operatorname{Ind}_{\Gamma_0}^\Lambda (\pi_0)$.  By general properties of induction 
(see for example \cite{Mackey}), we have that 
\begin{equation*}
\begin{aligned}
           \pi
\leq &\operatorname{Ind}_{\Gamma_0}^\Lambda(\pi_0)
\leq  \operatorname{Ind}_{\Gamma_0}^\Lambda
        \big(\operatorname{Ind}_{\Gamma_0\cap \Gamma_1}^{\Gamma_0}(\pi_0|_{\Gamma_0\cap \Gamma_1})\big)\\
     =& \operatorname{Ind}_{\Gamma_0\cap \Gamma_1}^\Lambda(\pi_0|_{\Gamma_0\cap \Gamma_1})
        = \operatorname{Ind}_{\Gamma_1}^\Lambda
     \big(\operatorname{Ind}_{\Gamma_0\cap \Gamma_1}^{\Gamma_1}(\pi_0|_{\Gamma_0\cap \Gamma_1})\big)\,.
\end{aligned}
\end{equation*}
By Theorem~\ref{thm:stab}(1) we know that 
$\pi_0|_{\Gamma_0\cap \Gamma_1}\in\mathbf{Mult}(\Gamma_0\cap \Gamma_1)$ 
and hence, by Theorem~\ref{thm:stab}(2), 
$\operatorname{Ind}_{\Gamma_0\cap \Gamma_1}^{\Gamma_1}(\pi_0|_{\Gamma_0\cap \Gamma_1})\in\mathbf{Mult}(\Gamma_1)$,
 so that
$\pi\in{\mathbf{Mult}_1(\Lambda)}$ and, by symmetry, ${\mathbf{Mult}_0(\Lambda)}={\mathbf{Mult}_1(\Lambda)}$.
\end{proof}

The above result justifies the following

\begin{definition}
We say that a representation $\pi$ of a virtually free group $\Lambda$ belongs
to the class ${\mathbf{Mult}(\Lambda)}$  if there exists a finite index
free subgroup $\Gamma\leq\Lambda$ and a representation
$\pi'$ in the class $\mathbf {Mult}(\Gamma)$ such that
$\pi$ is a component of $\operatorname{Ind}_{\Gamma}^\Lambda (\pi')$,
\begin{equation*}
\begin{aligned}
{\mathbf{Mult}(\Lambda)}:=\big\{\pi\in\Lambda\st \exists\, \pi'\in\mathbf{Mult}(\Gamma)
\text{ for some free subgroup }\Gamma\leq\Lambda&\\
\text{ of finite index such that }
\pi\leq \operatorname{Ind}_{\Gamma}^\Lambda (\pi')
\big\}&\,.
\end{aligned}
\end{equation*}
\end{definition}

The representations in the class ${\mathbf{Mult}(\Lambda)}$ are not necessarily irreducible,
but are however finitely reducible, as the following proposition shows:

\begin{proposition}\label{prop:finite-reducibility}  Let $\Lambda_0$ be a subgroup of
finite index of a group $\Lambda$ and let $\pi:\Lambda_0\to\mcU(\mcH)$ be an irreducible representation. 
Then $\big(\operatorname{Ind}_{\Lambda_0}^\Lambda(\pi),\operatorname{Ind}_{\Lambda_0}^\Lambda(\mcH)\big)$ 
is a finite sum of irreducible representations.
\end{proposition}

\begin{proof}  Let us denote $\rho:=\operatorname{Ind}_{\Lambda_0}^\Lambda(\pi)$ and 
$\mcL:=\operatorname{Ind}_{\Lambda_0}^\Lambda(\mcH)$.  Recall that 
\begin{equation*}
\mcL
:=\big\{f:\Lambda\to\mcH:\,\pi'({\gamma_0})f(\gamma)=f(\gamma\inv{{\gamma_0}}),\text{ for all }{\gamma_0}\in{\Lambda_0},\gamma\in\Lambda\big\}
\end{equation*}
on which $\Lambda$ acts by
\begin{equation*}
\big(\rho(\gamma)f\big)(\eta):=f(\inv{\gamma}\eta)
\end{equation*}
for all $\eta,\gamma\in\Lambda$.  
The fact that  ${\Lambda_0}$ is of finite index in $\Lambda$, namely $\Lambda=\sqcup_{u\in D}u{\Lambda_0}$, 
where $D$ is a finite set of representatives, induces a finite decomposition
\begin{equation}\label{eq:dsd}
\mcL=\bigoplus_{u\in D}\mcL_u\,,
\end{equation}
where
\begin{equation*}
\mcL_u:=\big\{f\in\mcL:\\,\operatorname{supp}(f)\subset u{\Lambda_0}\big\}\,.
\end{equation*}
It is immediate to verify that for all $\eta\in\Lambda$ and $u\in D$,
one has that $\rho(\eta)\mcL_u\subseteq\mcL_{\eta u}$ and hence
\begin{equation*}
\rho(u{\gamma_0}\inv{u})\mcL_u\subseteq\mcL_u
\end{equation*}
for all ${\gamma_0}\in{\Lambda_0}$.
Moreover for all $u\in D$, the evaluation operator 
\begin{equation*}
\begin{aligned}
E_u:\mcL_u&\to\,\,\mcH\\
f\,\,&\mapsto f(u)
\end{aligned}
\end{equation*}
is a unitary isomorphism with the property that
\begin{equation*}
\pi({\gamma_0})\,E_u=E_u\,\rho(u{\gamma_0}\inv{u})\,,
\end{equation*}
for all ${\gamma_0}\in{\Lambda_0}$ and $u\in D$.  In other words, $E_u$ is an intertwining operator
between $(\pi,\mcH)$ and 
$(\rho|_{u{\Lambda_0}\inv{u}},\mcL_u)$.  Since $(\pi,\mcH)$
is irreducible, $(\rho|_{u{\Lambda_0}\inv{u}},\mcL_u)$ is irreducible as well.

Let now $T:\mcL\to\mcL$ be an intertwining operator for $\rho$.  If $p_u:\mcL\to\mcL_u$
is the orthogonal projection, then, for all $u, v\in D$, $p_v\,T\, p_u$ intertwines
$(\rho|_{(v{\Lambda_0}\inv{v})\cap(u{\Lambda_0}\inv{u})},\mcL_v)$ and 
$(\rho|_{(v{\Lambda_0}\inv{v})\cap(u{\Lambda_0}\inv{u})},\mcL_u)$.
Since $(v{\Lambda_0}\inv{v})\cap(u{\Lambda_0}\inv{u})$ is of finite index both in 
$v{\Lambda_0}\inv{v}$ and in $u{\Lambda_0}\inv{u}$, each of the above representations is finitely reducible, 
\cite{Poguntke}.
Hence the space of intertwining operators between $(\rho|_{(v{\Lambda_0}\inv{v})\cap(u{\Lambda_0}\inv{u})},\mcL_v)$ 
and 
$(\rho|_{(v{\Lambda_0}\inv{v})\cap(u{\Lambda_0}\inv{u})},\mcL_u)$ is finite dimensional, 
which forces the space of intertwining operators of $(\rho,\mcL)$ to be finite dimensional as well.
\end{proof}

\begin{definition}\label{defi:irr}  We say that a representation $\pi$ of $\Lambda$ belongs
to the class ${\mathbf{Mult}_\mathrm{irr}(\Lambda)}$  if there exists a finite index
free subgroup $\Gamma\leq\Lambda$ and a representation
$\pi'$ in the class $\mathbf {Mult}(\Gamma)$ such that
$\pi$ is an irreducible component of $\operatorname{Ind}_{\Gamma}^\Lambda (\pi')$,
\begin{equation*}
\begin{aligned}
{\mathbf{Mult}_\mathrm{irr}(\Lambda)}:=\big\{\pi\in\Lambda\st \exists\, \pi'\in\mathbf{Mult}(\Gamma)
\text{ for some free subgroup }\Gamma\leq\Lambda&\\
\text{ of finite index such that }
\pi\leq \operatorname{Ind}_{\Gamma}^\Lambda (\pi')
\text{ and }\pi\text{ is irreducible}
\big\}&\,.
\end{aligned}
\end{equation*}
\end{definition}

The fact that this class is not empty follows from Proposition~\ref{prop:finite-reducibility} and the fact
that, by Theorem~\ref{thm:decomposition}, any representation in the class $\mathbf{Mult}(\Gamma)$ 
is a finite sum of irreducible representations in the same class.

\begin{corollary}\label{cor:invariance}  For a finitely generated virtually  free group $\Lambda$
the classes ${\mathbf{Mult}(\Lambda)}$ and 
${\mathbf{Mult}_\mathrm{irr}(\Lambda)}$ are $\Aut(\Lambda)$-invariant.
\end{corollary}
\begin{proof} Let $\alpha\in\Aut(\Lambda)$, let $\Gamma<\Lambda$ be a free subgroup of finite
index and let $\pi\in\mathbf{Mult}(\Gamma)$. 
For $\gamma\in\alpha(\Gamma)$ set $\pi^\alpha(\gamma):=\pi(\inv\alpha\gamma)$.
An easy verification shows that
\begin{equation*}
\operatorname{Ind}_{\alpha(\Gamma)}^\Lambda(\pi^\alpha)
\simeq\operatorname{Ind}_{\Gamma}^\Lambda\big(\pi\big)\circ\alpha\,.
\end{equation*}
The fact that  $\pi^\alpha\in \mathbf{Mult}(\alpha(\Gamma))$ 
(\cite{Iozzi_Kuhn_Steger_stab}) and Proposition~\ref{prop:independence}
show the assertion.
\end{proof}

We may then conclude:
\begin{corollary}\label{cor:bdry-rep-Gamma} The representations of a finitely generated 
virtually free group $\Lambda$ in the class  ${\mathbf{Mult}(\Lambda)}$ 
(and hence ${\mathbf{Mult}_\mathrm{irr}(\Lambda)}$)
are weakly contained in the regular representations.
\end{corollary}
\begin{proof} Since  representations in the class $\mathbf{Mult}{ }$ of the free group are weakly
contained in the regular representation \cite{K-S1}
the continuity of the induction map ensures that every representation
in the class ${\mathbf{Mult}(\Lambda)}$
is weakly contained in the regular representation of $\Lambda$.
\end{proof}

\section{Tempered Representations of Gromov Hyperbolic Groups }\label{sec:result-hyperbolic}
In this section we prove further properties of the representations in
the class ${\mathbf{Mult}_\mathrm{irr}(\Lambda)}$, namely that they
can be extended to boundary representations (Theorem~\ref{general}).
This will follow from general arguments in operator algebras which
hold for general Gromov hyperbolic groups and do not depend on the
particular construction of the class
${\mathbf{Mult}_\mathrm{irr}(\Lambda)}$, but rather only on the fact
that the representations in the class
${\mathbf{Mult}_\mathrm{irr}(\Lambda)}$ are tempered.  In this section
$G $ is a Gromov hyperbolic group.

\medskip
We saw already that boundary representations are associated with the action of $G $ on its
boundary $\partial G $ and we mentioned that they are in fact 
representations of the  crossed product $G\ltimes \mcC({\partial G})$. 
We recall here the definitions that will be needed for the proof of the next theorem
(and at the same time clarify the above assertions).

\smallskip
Let $\mcA$ be a $C^\ast$-algebra and let us denote by $\mcA[G]$ the space of finitely
supported functions $G\to\mcA$, 
\begin{equation*}
\mcA[G]:=\bigg\{\sum_i \zeta_i\delta_{\gamma_i}:\, \zeta_i\in \mcA, \gamma_i\in G\bigg\}\,,
\end{equation*}
where $\delta_\gamma$ is the Kronecker function at $\gamma\in G $.
If $ G $ acts on $\mcA$ by isometric automorphisms ${\lambda}: G \to\Aut(\mcA)$,
we endow $\mcA[ G ]$ with a $C^\ast$-algebra structure as follows.
Define the sum of two elements of $\mcA[ G ]$ in the obvious way (as
$\mcA$-valued functions on $ G $) and let
\begin{equation}\label{prodotto}
(\zeta_1\delta_{\gamma_1})\cdot(\zeta_2\delta_{\gamma_2}):=
\big(\zeta_1{\lambda}({\gamma_1})\zeta_2\big)\delta_{\gamma_1\gamma_2}\period
\end{equation}
Use the distributive law to extend 
\eqref{prodotto} to a product on $\mcA[ G ]$. 
Finally set 
\begin{equation*}
(\zeta\delta_\gamma)^\ast:= \lambda(\inv\gamma)
  \zeta^*\delta_{\inv\gamma}\period
\end{equation*}

In order to define a norm on $\mcA[ G ]$, take any covariant
representation $(\pi,\alpha,H)$ of $( G ,\mcA)$  and
for $f=\sum_i \zeta_i\delta_{\gamma_i}\in\mcA[ G ]$ 
define the operator 
\begin{equation*}
(\pi\ltimes\alpha)(f):=\sum_i
\alpha(\zeta_i)\pi(\gamma_i)\,.
\end{equation*}
Define now the {\em universal norm}
\begin{equation}\label{norma}
\|f\|:=\sup\|(\pi\ltimes\alpha)(f)\|_H\,,
\end{equation}
where the supremum is taken over all covariant representations $(\pi,\alpha,H)$ of $ G $.
The completion of $\mcA[ G ]$ with respect to the above norm is
the {\em (full) crossed product $C^\ast$-algebra $ G \ltimes\mcA$}.

Given a $C^\ast$-representation $\alpha$ of $\mcA$ on $H$, 
one can always get a covariant representation $(\tilde\lambda,\tilde\alpha)$ of $( G ,\mcA)$ 
on $\ell^2( G )\otimes H$ by setting
\begin{align*}
\big(\tilde\alpha(\zeta)\xi\big)(\gamma)&:=\alpha(\lambda(\inv\gamma)\zeta)\xi(\gamma)\\
\big(\tilde\lambda(\gamma')\xi\big)(\gamma)&:=\xi(\inv{\gamma'}\gamma)\,,
\end{align*}
for all $\zeta\in\mcA$, $\gamma,\gamma'\in G $ and $\xi\in\ell^2( G )\otimes H$.
We remark, for further purposes,  that $\tilde\lambda$ consists of $d$ copies
of the regular representation $\pi_{\text{reg}}$ of $ G $, 
where $d$ is the Hilbert dimension of $H$.
The completion of $\mcA[ G ]$ with respect to the {\em reduced norm} 
\begin{equation*}
\|f\|_{\text{red}}:=\sup_\alpha\|(\tilde\lambda\ltimes\tilde\alpha)(f)\|_{\ell^2( G )\otimes H}\,,
\end{equation*}
where the supremum in \eqref{norma} is taken  only
over those covariant representations of the form
$(\tilde\lambda,\tilde\alpha)$
is the {\em reduced crossed product $C^\ast$-algebra} $ G \ltimes_{\text{red}} \mcA$.

\begin{example} The examples of this construction relevant to our purposes are the following:
\begin{itemize}
\item $\mcA=\mathbf{C}$ is the $C^\ast$-algebra of complex numbers with the trivial $ G $-action;
 in this case $ G \ltimes\mathbf{C}$ is called the {\em group $C^\ast$-algebra}, denoted by $C^\ast( G )$,
 and $ G \ltimes_{\text{red}}\mathbf{C}$ is called the {\em reduced group $C^\ast$-algebra}, 
 denoted by $C^\ast_{\text{red}}( G )$.
 \item  $\mcA=\mcC(\partial G )$ is the $C^\ast$-algebra  of continuous functions on the boundary $\partial G $ of $ G $.
\end{itemize}
\end{example}
\smallskip

We conclude this discussion by exhibiting a universal construction for
representations of the cross product $ G \ltimes\mcC(\partial G )$
\cite[Chapter~X, Theorem~3.8]{T} . Such representations are  also called
{\em cocycle representations} (see for instance the papers of C.~Anantharaman \cite{An2}
and of C.~Anatharaman and J.~Renault \cite{An-R}) and
also appear in the context of  measured semidirect product groupoids (see \cite{Re}).

Let $\om\to H_\om$ be  a Borel field of Hilbert spaces 
and $\mu$ a quasi-invariant measure on ${\partial G }$. Denote by
$\mcH:=\int^\oplus_{\partial G } H_\om d\mu(\om)$ the direct integral and
by $P(\om,\gamma):=\frac{d\mu(\inv\gamma \om)}{d\mu(\om)}$ the Radon--Nikodym
cocycle of the $ G $ action.
For $\om_1$ and $\om_2$ in ${\partial G }$, denote by $\operatorname{Iso}
(H_{\om_1},H_{\om_2})$ the 
space of all isomorphisms from $H_{\om_1}$ to $H_{\om_2}$. 
A unitary Borel cocycle is a map $A(\om,\gamma):\partial G \times G \to
\operatorname{Iso}(H_{\inv \gamma\om}, H_\om)$ such that 
\begin{itemize}
\item $A(\om,\gamma_1\gamma_2)=A(\om,\gamma_1)A(\inv{\gamma_1}\om,\gamma_2)$ [$\mbox{a.e.}\mu$], and
\item the map
$ \om\to\big\langle f(\om),A(\om,\gamma)g(\inv\gamma\om)\big\rangle$
is measurable for every pair of elements $f$, $g\in\mcH$.
\end{itemize}
\begin{definition}
A cocycle representation $\pi$ of $ G $
is a unitary representation acting on $\mcH$ according to the rule
\begin{equation*}
\big(\pi(\gamma)f\big)(\om):=P^{\frac12}(\om,\gamma)A(\om,\gamma)f(\inv\gamma\om)\;.
\end{equation*}
\end{definition}

\medskip
We can now prove the following:

\begin{theorem}\label{general}
Let $ G $ be a torsion free Gromov hyperbolic group which is not almost cyclic. 
Then every tempered representation of $ G $ is a cocycle representation with respect to 
some quasi-invariant measure.  

If the representation is irreducible, the measure 
can be taken to be ergodic and hence 
the dimension of $H_\omega$ is almost everywhere constant.
\end{theorem}

\begin{proof}  The inclusion $\mathbf{C}\hookrightarrow\mcC(\partial G )$
induces a map
$\phi:\mathbf{C}[ G ]\to\mcC({\partial G })[ G ]$ 
defined by 
\begin{equation*}
\phi\bigg(\sum_i\zeta_i\delta_{\gamma_i}\bigg):= 
\sum_i\zeta_i\one_{\partial G }\delta_{\gamma_i}\,,
\end{equation*}
where $\one_{\partial G }\in\mcC(\partial G )$ denotes the function identically one on ${\partial G }$.
It is immediate to verify that $\phi$ is continuous with respect to the reduced norm on both sides: in fact,
since  $\alpha(\one_{\partial G })$ is the identity operator , then
\begin{equation*}
\begin{aligned}
  \bigg\|\phi\bigg(\sum_i\zeta_i\delta_{\gamma_i}\bigg)\bigg\|_{\text{red}}
&=\sup_\alpha\bigg\|(\tilde\lambda\ltimes\tilde\alpha)\bigg(\sum_i\zeta_i\one_{\partial G }\delta_{\gamma_i}\bigg)\bigg\|_{\ell^2( G )\otimes H}\\
&=\bigg\|\tilde\lambda\bigg(\sum_i\zeta_i\delta_{\gamma_i}\bigg)\bigg\|_{\ell^2( G )\otimes H}\\
&=\big\|\pi_{\text{reg}}\big(\sum_i\zeta_i\delta_{\gamma_i}\big)\big\|_{\ell^2( G )}\\
&=\bigg\|\sum_i\zeta_i\delta_{\gamma_i}\bigg\|_{\text{red}}\,.
\end{aligned}
\end{equation*}
Since the reduced $C^\ast$-algebra of $ G $ is simple  \cite{dLH}
(see also \cite{B-C-D} concerning lattices in semisimple Lie groups)
the extension of the above map $\overline\phi$ is actually an inclusion
\begin{equation}\label{map}
\overline\phi:C^\ast_{\text{red}}( G )\hookrightarrow 
 G \ltimes_{\text{red}}\mcC({\partial G })\;.
\end{equation}
Moreover, since the action of $ G $ on $\partial G $ is amenable (see \cite{Adams} or the more recent 
\cite{Ka}), then the reduced crossed product and the full crossed product 
coincide (see \cite[Theorem~5.3]{An1})) and hence we have
\begin{equation}\label{incl}
\overline{\phi}:C^\ast_{\text{red}}( G )\hookrightarrow G \ltimes\mcC({\partial G })\,.
\end{equation}

Assume now that $\pi$ is tempered or,
equivalently, that $\pi$ is a representation of $C^\ast_{\text{red}}( G )$.  
By  standard arguments involving the Hahn--Banach Theorem 
(see \cite[Lemma 2.10.1]{Dix}) one can see that $\pi$ can be extended to a representation
$\pi^{\partial G }$ of  $ G \ltimes\mcC({\partial G })$.  
By \cite[Chapter~X, Theorem~3.8]{T} 
the representations of the full crossed product are exactly 
the cocycle representations for some quasi-invariant measure $\mu$ on $\partial G $ 
and some field of Hilbert spaces $\omega\to H_\omega$.

The same argument in \cite[Lemma 2.10.1]{Dix} shows that  if $\pi$ is irreducible
one can require the extension  $\pi^{\partial G }$ to be also irreducible.
Since  $\pi^{\partial G }$ is irreducible, the corresponding measure $\mu$
is ergodic and, since the map $\omega\mapsto\dim(H_\omega)$ is
 measurable and $ G $-invariant,
the dimension of the Hilbert spaces $H_\om$ is constant $[a.e.\mu]$.
\end{proof}

\begin{remark}
The existence of the map \eqref{map}, and hence of the inclusion \eqref{incl},
is independent of the representation $\pi$ and depends only on the compactness of $\partial G $ and
the amenability of the $ G $-action.  Had we inputted in the picture from the beginning 
a representation $\pi$ that is weakly contained in the regular representation of 
$ G $, we could have obtained directly the map \eqref{incl}.  In fact, since 
cocycle representations are exactly the representations
of the full crossed-product $C^\ast$-algebra  and since the action of $ G $ on 
$\partial G $ is amenable, we have that the restriction of a representation of $ G \ltimes \mcC(\partial G )$
to $ G $ is weakly contained in the regular representation, 
that is $\|\pi(f\ltimes\one_{\partial G })\|\leq\|\pi_{\text{reg}}(f)\|$, 
\cite{Ku}, and hence the continuity of the map
into $ G \ltimes\mcC(\partial G )$ is proved at once. 
\end{remark}

\begin{corollary}\label{cor:bdry}
Let $\Lambda$ be a finitely generated virtually free group and let $\pi$ be a representation 
in the class ${\mathbf{Mult}_\mathrm{irr}(\Lambda)}$.
Then $\pi$ is a cocycle representation with respect to
a quasi-invariant ergodic  measure $\mu$ on $\partial\Lambda$.
\end{corollary}
\begin{proof} Theorem~\ref{general} and Corollary~\ref{cor:bdry-rep-Gamma}.
\end{proof}

\begin{remark}
Theorem 4.3 states that every tempered irreducible
representation $(\pi,H)$ of a Gromov hyperbolic group $ G $ admits
at least one extension to an irreducible representation 
$(\pi^{\partial G },H_{\partial G })$ of $ G \ltimes_{\text{red}}\mcC({\partial G })$.
We call such an extension a {\it boundary realization for $\pi$}. 
We say moreover that a boundary realization is {\it perfect} if one can take
$H=H_{\partial G }$.

Even if there  is no {\it a priori} 
reason for $(\pi^{\partial G },H_{\partial G })$ to be unique, 
we have noticed that this is the case when $ G =\Gamma$ is a free group
and $\pi$ is a representation of the class $\mathbf{Mult}(\Gamma)$ whose matrix coefficients
are sufficently ``big'' in the sense of \cite{K-S2}. 
In fact for all irreducible tempered representations $(\pi,H)$ of the free group known so far, there are only 
three possibilities:
\begin{itemize}
\item $\pi$ admits only one boundary realization which is perfect.  In this case
the irreducible representation $(\pi^{\partial\Gamma},H)$ of $\Gamma\ltimes\mcC(\partial\Gamma)$
remains irreducible also when restricted to $\Gamma$;
in this case we say that $\pi$ satisfies {\it monotony}.
\item $\pi$ admits only one boundary realization which is not perfect,
so that the inclusion $H\hookrightarrow H_{\partial\Gamma}$ is proper.
In this case the representation $(\pi^{\partial\Gamma},H_{\partial\Gamma})$ is irreducible as representation  
of  $\Gamma\ltimes_{\text{red}}\mcC({\partial\Gamma})$, but, when restricted to 
$\Gamma$, it splits into the sum of two irreducible inequivalent representations;
we say that $\pi$ satisfies {\it oddity}.
\item $\pi$ admits exactly two perfect boundary realizations, no other 
boundary realization is perfect and any other (not perfect) boundary realization can be obtained 
as a linear combination of these two perfect ones;
in this last case we say that $\pi$ satisfies {\it duplicity}.
\end{itemize}
We have conjectured that those are the only three possibilities for
any tempered representation of a free group, 
but we can prove it so far only for representations of the class $\mathbf{Mult}(\Gamma)$, \cite{K-S2}.
We think that the same problem is well posed also for a Gromov hyperbolic group and 
perhaps passing to a more general class of groups will give a better understanding of this  phenomenon.
\end{remark}

As a consequence of Theorem~\ref{general}  we can state an analogue of
Herz majorization principle for a class of hyperbolic groups.

\begin{theorem}\label{prop:herz}
Let  $(\pi,H)$ be a tempered 
representation of  a torsion free Gromov hyperbolic group $ G $  
which is not almost cyclic  
and let $v$ be any vector in $H$. 
Then there exists a positive measure $\mu$ on $\partial G $ such that
\begin{equation}\label{eq:herz}
|\langle\pi(x)v,v\rangle|\leq \norm{v}^2|\langle\rho(x)\one_{\partial G },\one_{\partial G }\rangle|
\end{equation}
where $\rho$ is the quasi-regular representation on $L^2(\partial G ,d\mu)$ and 
$\one_{\partial G }$ is the constant function on $\partial G $.
\end{theorem}

\begin{proof}
Let $\pi^{\partial G }$  be a boundary representation extending $\pi$. 
Choose any element $v$  of norm one in $H$ and 
let $\om\mapsto f(\om)$ be the element of $\int_{\partial G }^\oplus H_\om d\mu(\om)$, corresponding to it. 
(We remark that in this case $\mu$ need not be ergodic.)
Let 
\begin{equation*}
E(f)=\{~\om~\in{\partial G }\st f(\om)\neq0\}\,,
\end{equation*} 

\begin{equation*}
F(\om):=\begin{cases}
\norm{f(\om)}\qquad&\text{if $\om\in E(f)$}\\
1\qquad& \text{if }\om\in{\partial G }\setminus E(f)\,,
\end{cases}
\end{equation*}
and set $dm(\om)={[F(\om)]^2d\mu(\om)}$. Since the map
$g(\om) \mapsto \frac{g(\om)}{F(\om)}$
is a unitary  equivalence between 
$\int_{\partial G }^\oplus H_\om d\mu(\om)$ and $\int_{\partial G }^\oplus H_\om dm(\om)$, we may
assume that $\norm{f(\om)}=1$ on $E(f)$. 
Denote by $P(x,\om)$ the Radon--Nikodym derivative of $\mu$ with respect to the 
$ G $ action, so that 
\begin{equation*}
(\pi^{\partial G }(x)f)(\om)=P^{\frac12}(x,\om)A(x,\om)f(\inv x\om)
\end{equation*}
for some unitary Borel cocycle $A(x,\om)$.
One has
\begin{equation*}
\begin{aligned}
|\langle\pi(x)v,v\rangle|=&|\int_{\partial G }\langle A(x,\om)f(\om),f(\om)\rangle_\om
	P^{\frac12}(x,\om)d\mu|\\
&\leq\int_{E(f)}\norm{f(\om)}_\om^2P^{\frac12}(x,\om)d\mu\\
&\leq\int_{\partial G } P^{\frac12}(x,\om)d\mu\\
&=\langle\rho(x)\one_{\partial G },\one_{\partial G }\rangle
\end{aligned}
\end{equation*}
where $\rho$ is the quasi-regular representation on $L^2({\partial G },d\mu)$.
\end{proof}

\begin{remark}\label{rem:4.8}
If $H$ is a semisimple Lie group with finite center and maximal compact $K$,
there exists a unique $K$-invariant probability measure $\mu$  on the 
maximal Furstenberg boundary $H/P$, for $P$ a minimal parabolic.
In this case the quasi-regular representation
$\rho$ on $L^2(H/P,d\mu)$ plays a very important role, namely 
the Harish-Chandra function 
$\Xi(x)=\langle\rho(x)\one_{H/P },\one_{H/P}\rangle$ dominates
all spherical functions associated with tempered unitary representations.
If $H$ has property (T) one can push this further by exhibiting a positive definite
function $\Psi$ which dominates all positive definite non-constant spherical functions on $H$.
R.~Howe and E.~C.~Tan constructed in their book \cite[Chapter V]{HT} 
such a function $\Psi$ from $\Xi$ for $SL(n,\mathbb R)$, for $n\geq3$,
while the more recent paper of H.Oh \cite{Oh} treats the general case.

\medskip 
We remark that the measure $\mu$ of Proposition~\ref{prop:herz}
must depend on $\pi$, making our case much more similar to $SL(2,\mathbb R)$,
for which it is impossible to bound an arbitrary matrix coefficient in terms
of $\Xi$.
To see this, take $\Gamma$ to be a non-abelian free group and take
a copy of $\mathbb Z$ inside $\Gamma$. Let $w$ be the generator for $\mathbb Z$, 
$\pi_{\mathbb Z}$  be the representation induced from the trivial character of $\mathbb Z$
and $\one_{[\mathbb Z]}$ the characteristic function of the coset $[\mathbb Z]$.
If there were to exist a fixed
measure $\mu$ such that \eqref{eq:herz} holds for every tempered $\pi$,
one would have
\begin{equation*}
\langle\rho(w)\one_{\partial\Gamma},\one _{\partial \Gamma}\rangle\geq
\langle\pi_{\mathbb Z}(w)\one_{[\mathbb Z]},\one _{[\mathbb Z]}\rangle
=1\period
\end{equation*}
The fact that every word $w$ generates a copy of $\mathbb Z$ inside $\Gamma$,
would then imply  that 
$\langle\rho(\,\cdot\,)\one_{\partial\Gamma},\one _{\partial \Gamma}\rangle\equiv1$
identically on $\Gamma$, which is impossible since the measure $\mu$ on $\partial\Gamma$ 
cannot be invariant.
\end{remark}

\appendix

\section{Boundaries}\label{sec:boundaries}
Fix   a generator system $S$ for a Gromov hyperbolic group  
$ G $ and denote by $X$ 
its Cayley graph with respect to $S$.
Then $X$ is a hyperbolic geodesic space with respect to the word metric $d_X$. 

Fix, once and for all, a base point $p\in X$.
A sequence of points $\{x_j\in X\}$ is said to {\em tend to infinity} if 
\begin{equation}\label{eq:inf}
\lim_{i,j\to\infty}(x_i|x_j)_p=+\infty\;,
\end{equation}
where $(x|y)_p$ is the Gromov product defined by 
\begin{equation*}
(x|y)_p:=\frac12\big\{d_X(x,p)+d_X(y,p)-d_X(x,y)\big\}\;,
\end{equation*}
for all $x,y,p\in X$.
It can be proved that \eqref{eq:inf} does not depend on the choice 
of the basepoint $p$.
Denote by $S_\infty$ the set of all sequences in $X$ tending to infinity.
Two sequences $\{x_j\}$ and  $\{y_j\}$ in $S_\infty$ are {\em equivalent} if 
\begin{equation*}
\lim_{j\to\infty}(x_j|y_j)_p=+\infty\,.
\end{equation*}
 It can be proved that this is a true
equivalence relation. The {\em boundary at infinity $\partial X$} 
of $X$ is the set
of all equivalence classes of sequences tending to infinity.
When a sequence $\{x_j\}$ represents a class $\om\in\partial X$, we 
say that $x_j$ converges to $\om$. 

An equivalent definition of notion of boundary of a hyperbolic group can be given as follows.
A {\em geodesic ray} is an isometric embedding $r:[0,+\infty)\to X$
of $\mathbf R^+$ into $X$. 
Given a geodesic ray, there exists a unique $r(\infty)\in\partial X$
such that $r(t_j)$ converges to $r(\infty)$ for every sequence of 
real points $\{t_j\}$ going to $+\infty$.

Denote by $R_p$ the set of all  geodesic
rays starting at $p$ ($r(0)=p$). 
Two rays $r$ and $r'$ are equivalent ($r\sim r'$)
if 
\begin{equation*}
d_X\big(r(t),r'(t)\big)\, \text{ is bounded as }t \to\infty\,.
\end{equation*}
The quotient set $R_p/\sim$ is called the
 {\em visual boundary} of $X$ and it can be proved that it does not depend
on the choice of $p$. 
Since $R_p$ has a topology derived from the uniform convergence on compact
intervals of geodesic rays, we endow $R_p/\sim$ with its quotient topology.
Since $ G $ is finitely generated 
(see \cite{delaharpe_ghys}) every closed ball is finite (hence compact!)
and so $X$ is a {\em proper} geodesic space. By  
\cite[Proposition 2.64]{O} the {\em visual boundary} is compact and coincides with the
boundary at infinity defined above, so that we shall denote
by $\partial G$ any of these two boundaries.
The action of $ G $ on $X$ extends in an obvious way to an action on
$\partial G$. 

\section{Proof of Theorem~\ref{thm:4.4}}\label{app:2}
As mentioned in \S~\ref{sec:stability}, the proof is just a straightforward
verification that however uses heavily all the operators and objects defined in 
\cite{Iozzi_Kuhn_Steger_stab}.  
We will hence show here only the following result; the other assertions of Theorem~\ref{thm:4.4} 
are left to the reader.

\begin{theorem}\label{thm:ind} Let $\Gamma_0 \leq \Gamma$ be  a subgroup of finite index 
in the free group $\Gamma$.
If $(\pi_0,\mcH_0)\in\mathbf{Mult}(\Gamma_0)$ is a boundary representation of $\Gamma_0$, then
the induced representation  $\ind{\pi_0}$ is a boundary representation
of $\Gamma$ in the class $\mathbf{Mult}(\Gamma)$.
\end{theorem}

We start by recalling some objects that were defined in \cite{Iozzi_Kuhn_Steger_stab} and 
that will be needed in the proof.  

Let $\mcT$ be the Cayley graph of the free group $\Gamma$ with respect to a symmetric set of free
generators $A$, and let $\Gamma_0\leq\Gamma$ be a subgroup. $\mcT$ is a tree in which we fix an origin $e$
and which is a metric space with the word distance with respect to the generating set $A$. It is always possible to choose a fundamental domain $D$
for the action of $\Gamma_0$ on $\mcT$ having the following properties:
\begin{itemize}
\item $D$ is a subtree containing $e$
\item $\Gamma_0$ is generated by the set
\begin{equation*}
A':=\big\{a'_j\in\Gamma_0:\,d(D,a'_jD)=1\big\}\period
\end{equation*}
\end{itemize}  
For any generator $a\in A$,
define the set 
\begin{equation*}
P(a):=(D^{-1}\cdot A')\cap\Gamma(a)\,,
\end{equation*}
where $\inv{D}=\big\{\inv{u}:\,u\in D\big\}$.

If $(V_{a'},H_{b'a'},B_{a'})$ is a matrix system with inner products for $\Gamma_0$ and $\pi_0$
is a representation in the class $\mathbf{Mult}(\Gamma_0)$ acting on the Hilbert space $\mcH_0:=\mcH(V_{a'},H_{b'a'},B_{a'})$,
we consider the induced representation $\ind{\pi_0}$ on $\ind{\mcH_0}$.

Via the assignment $f\mapsto\tilde f$, where $\tilde f(x):=f(u)(h)$ for $x=uh$, 
with $h\in\Gamma_0$ and $u\in D$, we have the identification
\begin{equation*}
\begin{aligned}
\ind{\mcH_0^\infty}&\cong\big\{\tilde f:D\cdot\Gamma_0\to\coprod_{a'\in A'} 
V_{a'}:\,\pi_0(h)\tilde f(g)=\tilde f(g\inv{h}),\\\text{ for all }&h\in\Gamma_0,g\in\Gamma\text{ and }
        \tilde f\text{ is multiplicative as a function of }\Gamma_0\big\}\,.
\end{aligned}
\end{equation*}
In \cite{Iozzi_Kuhn_Steger_stab}, $\big(\ind{\pi_0},\ind{\mcH_0}\big)$
is proved to be equivalent to a multiplicative
representation $(\pi,\mcH)$ on $\mcH:=\mcH(V_a,H_{ba},B_a)$.
The  spaces $V_a$ are indexed on pairs $(u,c')$ corresponding to elements
$u\in D$ and $c'\in A'$ such that $\inv u c'\in P(a)$ while the $H_{ba}$ 
are block matrices that will perform three kinds of operations
on a vector $w_a\in V_a$ with coordinates $w_a=(w_a)_{u,c'}$. 
We give for completeness the explicit expressions for  $V_a$, $H_{ba}$ and 
$B_a$,
but only the $V_a$ will be used in the sequel:
\begin{equation*}
V_a:=\bigoplus_{u\in D}\bigoplus_{\substack{c'\in A'\\ \inv{u}c'\in P(a)}}V_{c'}\,,
\end{equation*}
\begin{equation*}
(H_{ba}w_a)_{{v},d'}:=
\begin{cases}
\begin{aligned}
&(w_a)_{v\inv{a},d'}&\qquad\text{if }v\inv{a}\in D\\
&H_{d'c'}(w_a)_{u,c'} &\hphantom{XX}\text{if } v\inv{a}\notin D \text{ and }a^{-1}v=u^{-1}c'\\
&0&\text{otherwise}\,,
\end{aligned}
\end{cases}
\end{equation*}
and
\begin{equation*}
({B_a})_{{u},c'}:=B_{c'}\qquad\text{where $\inv{u}c'\in P(a)$}\,.
\end{equation*}
What will be important instead is the explicit form of the intertwining operator
\begin{equation*}
J:\operatorname{Ind}_{\Gamma_0}^\Gamma\big(\mcH(V_{a'},H_{b'a'},B_{a'})\big)
\to\mcH(V_a,H_{ba},B_a)\
\end{equation*}
defined as
\begin{equation}\label{eq:J}
Jf(xa):=\bigoplus_{(u,c')}f(x\inv{u})(c')
\end{equation}
on the dense subspace $\operatorname{Ind}_{\Gamma_0}^\Gamma\big(\mcH^\infty(V_{a'},H_{b'a'},B_{a'})\big)$ of 
$\operatorname{Ind}_{\Gamma_0}^\Gamma\big(\mcH(V_{a'},H_{b'a'},B_{a'})\big)$.

\begin{proof}[Proof of Theorem~\ref{thm:ind}] We assume the result of
  Theorem~\ref{thm:decomposition}, namely the independence of the
  generating set.

Let $\Gamma_0<\Gamma$ be a finite index free subgroup, 
$(\pi_0,\mcH_0)\in\mathbf{Mult}(\Gamma_0)$ a boundary representation of $\Gamma_0$ 
and let us consider as in \eqref{eq:bdryrep} be the associated boundary representation 
$(\pi_0,\alpha_{\pi_0},\mcH_0)$ of $\Gamma_0$.

Let $\lambda$ denote the isometric action both of $\Gamma$ on
$\mcC(\partial\Gamma)$ and of $\Gamma_0$ on
$\mcC(\partial{\Gamma_0})$.  If
$\varphi:\partial\Gamma_0\to\partial\Gamma$ is the boundary
homeomorphism, every function $F\in \mcC(\partial\Gamma)$ defines a
function $F\circ\varphi\in \mcC (\partial\Gamma_0)$,
and hence, if we define (with quite a dose of pedantry...) 
an action $\alpha$ of $\mcC(\partial\Gamma)$ on $\mcH_0$ by
\begin{equation}\label{alpha}
\alpha(F):=\alpha_{\pi_0}(F\circ\varphi)\,,
\end{equation}
it is straightforward to verify that $(\pi_0,\alpha,\mcH_0)$ is a
covariant representation of $\big(\Gamma_0,\mcC(\partial\Gamma)\big)$.

Let us now consider the induced representation $\piind:=\ind{\pi_0}$
on $\Hind:=\ind{\mcH_0}$.
Define an action of $\mcC(\partial\Gamma)$ on $\Hind$ by setting
\begin{equation}\label{eq:C(Om)onHind}
\big(\Pi(F){f}\big)(x):=\alpha\big(\lambda(\inv x)F\big){f}(x)\,,
\end{equation} 
for $f\in\Hind$, $F\in \mcC(\partial\Gamma)$ and $x\in\Gamma$.
We have to check that, under the above assumptions,
$\Pi(F){f}$ is still in $\Hind$, namely
\begin{equation}\label{eq:bdry-repr}
\big(\Pi(F){f}\big)(x \gamma)=\pi_0(\inv \gamma)\big(\Pi(F){f}\big)(x)\,.
\end{equation}
But this is straightforward as, by using \eqref{eq:C(Om)onHind}, 
the covariance of $(\pi_0,\alpha,\mcH_0)$ and \eqref{eq:bdry-repr},
we verify that
\begin{equation*}
\begin{aligned}
 \Pi(F){f}(x \gamma)
=&\alpha\big(\lambda\inv{(x \gamma)}F\big){f}(x \gamma)\\
=&\alpha\big(\lambda(\inv{\gamma})\lambda(\inv{x})F\big){f}(x \gamma)\\
=&\pi_0(\inv{\gamma})\alpha\big(\lambda(\inv{x})F\big)\pi_0(\gamma)f(x \gamma)\\
=&\pi_0(\inv{\gamma})\alpha\big(\lambda(\inv{x})F\big)f(x)\\
=&\pi_0(\inv{\gamma})\big(\Pi(F){f}\big)(x)\,.
\end{aligned}
\end{equation*}
In the same way one proves that
\begin{equation*}
\piind(\gamma)\Pi(F)\piind(\inv \gamma)=\Pi\big(\lambda(\gamma)F\big)\,
\end{equation*}
for all $x\in\Gamma$ and $F\in \mcC(\partial\Gamma)$, thus showing that  
$(\piind,\Pi,\Hind)$ is a boundary representation of $\Gamma$.

What is left to be shown is that if $J$ is the operator that
intertwines $(\piind,\Hind)$ and a multiplicative representation
$(\pi,\mcH)$ of $\Gamma$, then $J$ intertwines also
$\Pi:\mcC(\partial\Gamma)\to\mcL(\Hind)$ and
$\alpha_\pi:\mcC(\partial\Gamma)\to\mcL(\mcH)$ (defined respectively
in \eqref{eq:C(Om)onHind} and \eqref{eq:bdryrep}), namely that
\begin{equation}\label{intertw-cov}
J\,\Pi(F)=\alpha_\pi(F)\, J\,,
\end{equation}
for all $F\in \mcC(\partial\Gamma)$.
It will be indeed enough to verify that for all $f\in\mcH^\infty$ and $F\in \mcC(\partial\Gamma)$
\begin{equation*}
J\big(\Pi(F)f\big)=\alpha_\pi(F)J(f)\,.
\end{equation*}

Using the direct sum decomposition in \eqref{eq:dsd},
we assume first that $f$ is supported on the coset $\Gamma_0$ and that $F=\one_{\partial\Gamma(y)}$ for
some $y\in \Gamma$.
By definition of $J$ in \eqref{eq:J} , we have
\begin{equation}\label{pione}
J\big(\Pi(\one_{\partial\Gamma(y)})f\big)(xa)=\sum_{\inv u c'\in P(a)}\big(\Pi(\one_{\partial\Gamma(y)})f\big)(x\inv u)(c')\;,
\end{equation}
where by \eqref{eq:C(Om)onHind}
\begin{equation*}
\big(\Pi(\one_{\partial\Gamma(y)})f\big)(x\inv u)=\alpha\big(\lambda\inv{(u\inv x)}\one_{\partial\Gamma(y)}\big)f(x\inv u)\,.
\end{equation*}
Since $f$ is supported on $\Gamma_0$ the right hand side of
\eqref{pione} is zero unless
$x\inv u=\gamma\in \Gamma_0$: for these $x$ and $u$, by using the definition of $f$ and 
the covariance property of $(\pi_0,\alpha,\mcH_0)$, we have
 \begin{align*}
 \alpha\big(\lambda\inv{(u\inv x)}\one_{\partial\Gamma(y)}\big)f(x\inv u)
&=\alpha\big(\lambda(\inv \gamma\big)\one_{\partial\Gamma(y)})\pi_0(\inv\gamma)f(e)\\
&=\pi_0(\inv\gamma)\alpha(\one_{\partial\Gamma(y)})f(e)=\alpha(\one_{\partial\Gamma(y)})f(\gamma)\,.
\end{align*}
Substituting the result of these last two computations in \eqref{pione} we obtain
\begin{equation*}
\begin{aligned}
J\big(\Pi(\one_{\partial\Gamma(y)})f\big)(xa)
&=\sum_{\substack{x=\gamma u\\ \inv u c'\in P(a)}}
\big(\alpha(\one_{\partial\Gamma(y)})f(\gamma)\big)(c')\\
&=\sum_{\substack{x=\gamma u\\ \inv u c'\in P(a)}}
\alpha(\one_{\partial\Gamma(y)})\tilde f(\gamma c')=
\sum_{\substack{x=\gamma u\\ \inv u c'\in P(a)\\ \gamma c'\in \Gamma(y)}}\tilde f(\gamma c')\,,
\end{aligned}
\end{equation*}
where in the last equality we used the definition of $\alpha$ in \eqref{alpha} 
(and hence of $\alpha_{\pi_0}$ in \eqref{eq:bdryrep}).
We may assume that $|x|>|y|$, so that  $xa\in \Gamma(y)$
if and only if  $x\inv uc'=\gamma c'\in \Gamma(y)$; hence, by \eqref{eq:bdryrep},
\begin{equation*}
\begin{aligned}
\alpha_{\pi}(\one_{\partial\Gamma(y)})J(f)(xa)
=&\one_{\Gamma(y)}(xa)J(f)(xa)\\
=&\sum_{\substack{x=\gamma u\\ \inv u c'\in P(a)\\ \gamma c'\in \Gamma(y)}}\tilde f(\gamma c')
=J\big(\Pi(\one_{\partial\Gamma(y)}) f\big)(xa)\,,
\end{aligned}
\end{equation*}
which proves \eqref{intertw-cov} for all $f$ supported on $\Gamma_0$.  

Finally, if $f$ is supported on $u\Gamma_0$ for some $u\in D$ then,
applying \eqref{intertw-cov} to $\piind(\inv u)f$ (which is supported on $\Gamma_0$)
and using both the covariance of $(\piind,\Pi,\Hind)$ and of $(\pi,\alpha_\pi,\mcH)$
and the fact that $J$ intertwines $(\piind,\Hind)$ and $(\pi,\mcH)$, 
one can easily verify \eqref{intertw-cov} in general.
\end{proof}

\def\cprime{$'$}
\providecommand{\bysame}{\leavevmode\hbox to3em{\hrulefill}\thinspace}
\providecommand{\MR}{\relax\ifhmode\unskip\space\fi MR }
\providecommand{\MRhref}[2]{%
  \href{http://www.ams.org/mathscinet-getitem?mr=#1}{#2}
}
\providecommand{\href}[2]{#2}


\begin{thebibliography}{BCdlH94}

\bibitem[AD02]{An1}
C.~Anantharaman-Delaroche, \emph{Amenability and exactness for dynamical
  systems and their {$C^\ast$}-algebras}, Trans. Amer. Math. Soc. \textbf{354}
  (2002), no.~10, 4153--4178 (electronic).

\bibitem[AD03]{An2}
\bysame, \emph{On spectral characterizations of amenability}, Israel J. Math.
  \textbf{137} (2003), 1--33.

\bibitem[Ada94]{Adams}
S.~Adams, \emph{Boundary amenability for word hyperbolic groups and an
  application to smooth dynamics of simple groups}, Topology \textbf{33}
  (1994), no.~4, 765--783.

\bibitem[AR01]{An-R}
C.~Anantharaman and J.~Renault, \emph{Amenable groupoids}, Groupoids in
  analysis, geometry, and physics ({B}oulder, {CO}, 1999), Contemp. Math., vol.
  282, Amer. Math. Soc., Providence, RI, 2001, pp.~35--46.

\bibitem[BCdlH94]{B-C-D}
M.~Bekka, M.~Cowling, and P.~de~la Harpe, \emph{Some groups whose reduced {$C^
  *$}-algebra is simple}, Inst. Hautes \'Etudes Sci. Publ. Math. (1994),
  no.~80, 117--134 (1995).

\bibitem[BdlH97]{Burger_delaHarpe}
M.~Burger and P.~de~la Harpe, \emph{Constructing irreducible representations of
  discrete groups}, Proc. Indian Acad. Sci. Math. Sci. \textbf{107} (1997),
  no.~3, 223--235.

\bibitem[BM11]{Bader_Muchnik}
U.~Bader and R.~Muchnik, \emph{Boundary unitary
  representations---irreducibility and rigidity}, J. Mod. Dyn. \textbf{5}
  (2011), no.~1, 49--69.

\bibitem[Dix64]{Dix}
J.~Dixmier, \emph{Les {$C^{\ast} $}-alg\`ebres et leurs repr\'esentations},
  Cahiers Scientifiques, Fasc. XXIX, Gauthier-Villars \& Cie,
  \'Editeur-Imprimeur, Paris, 1964.

\bibitem[dlH88]{dLH}
P.~de~la Harpe, \emph{Groupes hyperboliques, alg\`ebres d'op\'erateurs et un
  th\'eor\`eme de {J}olissaint}, C. R. Acad. Sci. Paris S\'er. I Math.
  \textbf{307} (1988), no.~14, 771--774.

\bibitem[dlHG90]{delaharpe_ghys}
P.~{d}e~la Harpe and {\'E}.~Ghys (eds.), \emph{Sur les groupes hyperboliques
  d'apr\`es {M}ikhael {G}romov}, Progress in Mathematics, vol.~83, Birkh\"auser
  Boston Inc., Boston, MA, 1990, Papers from the Swiss Seminar on Hyperbolic
  Groups held in Bern, 1988.

\bibitem[HT92]{HT}
R.~Howe and E.-C. Tan, \emph{Nonabelian harmonic analysis}, Universitext,
  Springer-Verlag, New York, 1992, Applications of ${{\rm{S}}L}(2,{{\bf{R}}})$.

\bibitem[IKS]{Iozzi_Kuhn_Steger_stab}
A.~Iozzi, M.~A. Kuhn, and T.~Steger, \emph{Stability properties of
  multiplicative representations of the free group}, preprint, 2011.

\bibitem[Kai04]{Ka}
V.~A. Kaimanovich, \emph{Boundary amenability of hyperbolic spaces}, Discrete
  geometric analysis, Contemp. Math., vol. 347, Amer. Math. Soc., Providence,
  RI, 2004, pp.~83--111.

\bibitem[KPS73]{Karrass_Pietrowski_Solitar}
A.~Karrass, A.~Pietrowski, and D.~Solitar, \emph{Finite and infinite cyclic
  extensions of free groups}, J. Austral. Math. Soc. \textbf{16} (1973),
  458--466, Collection of articles dedicated to the memory of Hanna Neumann,
  IV.

\bibitem[KS96]{K-S1}
M.~G. Kuhn and T.~Steger, \emph{More irreducible boundary representations of
  free groups}, Duke Math. J. \textbf{82} (1996), no.~2, 381--436.

\bibitem[KS01]{K-S2}
\bysame, \emph{Monotony of certain free group representations}, J. Funct. Anal.
  \textbf{179} (2001), no.~1, 1--17.

\bibitem[KS04]{K-S3}
\bysame, \emph{Free group representations from vector-valued multiplicative
  functions. {I}}, Israel J. Math. \textbf{144} (2004), 317--341.

\bibitem[Kuh94]{Ku}
M.~G. Kuhn, \emph{Amenable actions and weak containment of certain
  representations of discrete groups}, Proc. Amer. Math. Soc. \textbf{122}
  (1994), no.~3, 751--757.

\bibitem[Mac76]{Mackey}
G.~W. Mackey, \emph{The theory of unitary group representations}, University of
  Chicago Press, Chicago, Ill., 1976, Based on notes by James M. G. Fell and
  David B. Lowdenslager of lectures given at the University of Chicago,
  Chicago, Ill., 1955, Chicago Lectures in Mathematics.

\bibitem[Oh02]{Oh}
H.~Oh, \emph{Uniform pointwise bounds for matrix coefficients of unitary
  representations and applications to {K}azhdan constants}, Duke Math. J.
  \textbf{113} (2002), no.~1, 133--192.

\bibitem[Ohs02]{O}
K.~Ohshika, \emph{Discrete groups}, Translations of Mathematical Monographs,
  vol. 207, American Mathematical Society, Providence, RI, 2002, Translated
  from the 1998 Japanese original by the author, Iwanami Series in Modern
  Mathematics.

\bibitem[Pog75]{Poguntke}
D.~Poguntke, \emph{Decomposition of tensor products of irreducible unitary
  representations}, Proc. Amer. Math. Soc. \textbf{52} (1975), no.~196,
  427--432.

\bibitem[Pow75]{Pow}
R.~T. Powers, \emph{Simplicity of the {$C\sp{\ast} $}-algebra associated with
  the free group on two generators}, Duke Math. J. \textbf{42} (1975),
  151--156.

\bibitem[Ren80]{Re}
Jean Renault, \emph{A groupoid approach to {$C^{\ast} $}-algebras}, Lecture
  Notes in Mathematics, vol. 793, Springer, Berlin, 1980. \MR{584266
  (82h:46075)}

\bibitem[Tak03]{T}
M.~Takesaki, \emph{Theory of operator algebras. {II}}, Encyclopaedia of
  Mathematical Sciences, vol. 125, Springer-Verlag, Berlin, 2003, Operator
  Algebras and Non-commutative Geometry, 6.

\end{thebibliography}

\end{document}